\documentclass[12pt]{article}
\usepackage{geometry}                % See geometry.pdf to learn the layout options. There are lots.
\usepackage{graphicx,color}
\usepackage{amssymb,amsmath,amsthm,mathrsfs}
\usepackage[all,cmtip]{xy}
\usepackage{epstopdf, comment}

\numberwithin{equation}{section}

\newtheorem{theorem}{Theorem}[section]

\newtheorem{conjecture}[theorem]{Conjecture}
\newtheorem{lemma}[theorem]{Lemma}
\newtheorem{corollary}[theorem]{Corollary}

\theoremstyle{remark}
\newtheorem*{remark}{Remark}

\theoremstyle{definition}

\DeclareMathOperator{\BP}{BP}
\DeclareMathOperator{\MBP}{MBP}

\DeclareMathOperator{\BMO}{BMO}
\DeclareMathOperator{\GCE}{GCE}
\DeclareMathOperator{\loc}{loc}

\DeclareMathOperator{\inn}{Inn}
\DeclareMathOperator{\out}{Out}

\DeclareMathOperator{\aut}{Aut}
\DeclareMathOperator{\hol}{Hol}

\DeclareMathOperator{\crit}{crit}

\title{Critical structures of inner functions}
\author{Oleg Ivrii}
\date{November 16, 2020}

\begin{document}

\maketitle

\begin{abstract}
A celebrated theorem of M.~Heins says that up to post-composition with a M\"obius transformation, a finite Blaschke product is uniquely determined by its critical points.  K.~Dyakonov suggested that it may interesting to extend this result to infinite degree, however, one needs to be careful since inner functions may have identical critical sets.
 In this work, we try parametrizing inner functions by 1-generated invariant subspaces of the weighted Bergman space $A^2_1$. Our technique is based on the Liouville correspondence which provides a bridge between complex analysis and non-linear elliptic PDE.
\end{abstract}

\section{Introduction}

A finite Blaschke product $F(z)$ is a holomorphic self-map of the unit disk which  extends to a continuous dynamical system on the unit circle. 
The most common way to study Blaschke products is by examining their zero sets.
It is not difficult to show that a finite Blaschke product $F(z)$ is uniquely determined by its zero set up to a rotation:
$$
F(z) = e^{i\theta} \prod_{i=1}^d \frac{z-a_i}{1-\overline{a_i}z}, \qquad a_1, a_2, \dots, a_d \in \mathbb{D},
$$
where $d \ge 1$ is  the degree of $F$.
This approach allows one to factor zeros of bounded analytic functions and leads to Beurling's invariant subspace theorem, which is one of the cornerstones of modern complex analysis and function theory.

In this work, we follow a less traveled path of examining the {\em critical sets} of Blaschke products, initiated by M.~Heins \cite{heins} in 1962.

\begin{theorem}[Heins]
\label{heins-thm}
A finite Blaschke product is uniquely determined by the set of its critical points up to post-composition with a M\"obius transformation
$m \in \aut(\mathbb{D})$, and furthermore, any set of $d-1$ points in the unit disk arises as the critical set of some Blaschke product of degree $d$.
\end{theorem}
 
Loosely speaking, an {\em inner function} is a holomorphic self-map of the unit disk which extends to a measure-theoretic dynamical system of the unit circle. More precisely, 
we want the radial boundary values to exist almost everywhere and have absolute value one. 

If one wants to generalize Heins' result to inner functions, one is confronted with the following obstacle:
different inner functions can have the same critical set. For example, $$F_1(z) = z, \qquad F_2(z) =
\exp \biggl (\frac{z+1}{z-1} \biggr )$$ have no critical points.
 In order to distinguish $F_1$ and $F_2$, one must record some additional information.
 In \cite{inner}, the author explained how to parametrize inner functions of finite entropy (with derivative in the Nevanlinna class), answering a question posed by K.~Dyakonov \cite{dyakonov-mobius, dyakonov-inner}:

\begin{theorem}
\label{main-thm}
Let $\mathscr J$ be the set of inner functions whose derivative lies in the Nevanlinna class. The natural map $$F \to \inn(F') \quad : \quad \mathscr J/\aut(\mathbb{D}) \to \inn/ \, \mathbb{S}^1$$
is injective. The image consists of all inner functions of the form $BS_\mu$ where $B$ is a Blaschke product and $S_\mu$ is the singular factor associated to a measure $\mu$ whose support is contained in a countable union of Beurling-Carleson sets.
\end{theorem}

 By definition, a {\em Beurling-Carleson set} $E \subset \partial \mathbb{D}$ is a closed subset of the unit circle of zero Lebesgue measure whose complement is a union of arcs $\bigcup_k I_k$ with $$\kappa(E) \, = \, \sum |I_k| \log \frac{1}{|I_k|} \, < \,  \infty.$$

\subsection{Beurling's theorem}

It would be desirable to parametrize all inner functions by their critical structure, not just the relatively small subset $\mathscr J$. To understand how this might be done, let us see how one can parametrize {\em zero structures} of inner functions. 
The following statement expresses the fact that zero sets of functions in the Hardy space $H^2$ are Blaschke sequences:
\begin{equation*}
 \BP / \, \mathbb{S}^1 = \{\text{zero-based subspaces of }H^2\}.
\end{equation*}
By definition, a (closed) subspace $X \subset H^2$ is {\em zero-based} if it is defined as the collection of functions in $H^2$ that vanish at a prescribed set of points.
Taking the ``closure'' of the above statement, we arrive at a famous theorem of Beurling:

\begin{theorem}[A.~Beurling, 1949]
\begin{align*}
 \inn / \, \mathbb{S}^1 & =\overline{ \{\text{\em zero-based subspaces of }H^2\}} \\
 & = \{\text{\em invariant subspaces of }H^2\}.
\end{align*}
\end{theorem}

The collection of closed subspaces of a Banach space carries a natural topology where  $X_n \to X$  if  for any convergent sequence $x_n \to x$ with $x_n \in X_n$, the limit $x \in X$, and conversely, any $x \in X$ can be approximated by a convergent sequence $x_n \to x$ with $x_n \in X_n$.  The process of taking closure has been given the beautiful name {\em asymptotic spectral synthesis} by N.~K.~Nikol'skii \cite[p.~34]{nikolskii}.

\subsection{Critical structures of inner functions}

Let $(H^2)'$ denote the space of  derivatives of $H^2$ functions. According to the Littlewood-Paley identity,
\begin{equation}
\label{eq:lp}
\| f \|^2_{H^2} = |f(0)|^2 + \frac{1}{2\pi} \int_{\mathbb{D}} |f'(z)|^2 \log\frac{1}{|z|} \, |dz|^2,
\end{equation}
one can identify $(H^2)'$  with the {\em  weighted Bergman space} $A_1 ^2$ which is the collection of all holomorphic function on the unit disk that satisfy
 \begin{equation*}
 \|f\|_{A_1^2} = \biggl (  \int_{\mathbb{D}} |f(z)|^2  (1-|z|) |dz|^2 \biggr)^{1/2} < \infty.
\end{equation*}
The starting point for our discussion is a beautiful result of Kraus \cite{kraus} which says that critical sets of Blaschke products coincide with critical sets of $H^2$ functions.
 One can tautologically rewrite Kraus' result as
$$
 \MBP / \, \aut(\mathbb{D}) = \{\text{zero-based subspaces of }A^2_1\}.
$$
The exact definition of the class of {\em maximal Blaschke products} will be given later. At the current moment, we only need to know that if $C$ is a critical set of some Blaschke product, then the  maximal Blaschke product with critical set $C$ is a particular Blaschke product with that critical set. 

For a function $H \in A^{2}_1$,  the {\em subspace generated by $H$} is defined as the minimal closed subspace of $A^2_1$ which contains $H$ and is invariant under multiplication by $z$:
$$
[H] = \overline{\{ Hp : p \text{ polynomial}\}}.
$$
In an important paper, S.~Shimorin \cite{shimorin} showed that the closure of the zero-based subspaces in $A^2_1$ consists of subspaces that can be generated by a single function. In light of Shimorin's result, we can try to write down the ``closure'' of Kraus' theorem:

\begin{conjecture}
\label{main-conjecture}
If an invariant subspace of $A^2_1$ can be generated by a single function, it can be generated by the derivative of an essentially unique inner function.
\end{conjecture}

While we are not able to fully resolve this conjecture, we can show that any 1-generated invariant subspace of $A^2_1$ can be generated by the derivative of a bounded function. Previously, it was known that the generator could be chosen to be the derivative of a $\BMO$ function, see \cite[Theorem 3.3]{HKZ}. 

\subsection{Liouville's theorem}
\label{sec:gce-connections}

Given a conformal pseudometric $\lambda(z)|dz|$ on the unit disk with an upper semicontinuous density, its {\em Gaussian curvature} is given by
$$
k_\lambda
=
 - \frac{\Delta \log \lambda}{\lambda^2},
 $$
where the Laplacian is taken in the sense of distributions.  It is well known that the Poincar\'e metric $\lambda_{\mathbb{D}}(z) = \frac{2}{1-|z|^2}$ has constant curvature $-1$, the Euclidean metric $\lambda_{\mathbb{C}}(z) = 1$ is flat, while the spherical metric $\lambda_{\widehat{\mathbb{C}}}(z) = 1$ has constant curvature $+1$. 

For a holomorphic self-map  of the unit disk $F \in \hol(\mathbb{D}, \mathbb{D})$, consider
the pullback
 $$
 \lambda_{F} := F^*\lambda_{\mathbb{D}} = \frac{2|F'|}{1-|F|^2}.
 $$
  Since curvature is a conformal invariant,
$
k_{\lambda_F} = -1
$
on $\mathbb{D} \setminus \crit(F)$ where  $\crit(F)$ denotes the critical set of $F$. However, on the critical set, $\lambda_F$ = 0 while its curvature has $\delta$-masses. Introducing the change of variables $u_F = \log \lambda_F$, we naturally arrive at the PDE
\begin{equation}
\label{eq:basic-gce}
   \Delta u = e^{2u} + 2\pi  \nu, \qquad  \nu \ge 0,
   \end{equation}
   where
   $ \nu = \sum_{c \in \crit(F)} \delta_{c}
   $ is an integral sum of point masses.
   
A theorem of Liouville  \cite[Theorem 5.1]{conf-metrics} states that the correspondence $F \to u_F$ is a bijection between
$$
\hol(\mathbb{D}, \mathbb{D})\,/\aut(\mathbb{D}) \quad \Longleftrightarrow \quad \bigl \{ \text{solutions of (\ref{eq:basic-gce}) with } \nu \text{ integral} \bigr \}.
$$
Liouville's theorem allows one to interpret properties of holomorphic self-maps of the disk as properties of solutions of the Gauss curvature equation. For example, Heins theorem states that a ``finite Blaschke product'' is a holomorphic self-map of the disk for which ``$u_F \to \infty$ as $|z| \to 1$'' while Theorem \ref{main-thm} states that an ``inner function of finite entropy''
corresponds to a ``{\em nearly-maximal solution} of the Gauss curvature equation'' for which
\begin{equation}
\label{eq:nm}
\limsup_{r \to 1} \int_{|z|=r} (u_{\mathbb{D}} - u) d\theta < \infty.
\end{equation}

\subsection{Weighted Gauss Curvature Equation}

In the present work, we will be mostly working with the weighted Gauss curvature equation $\Delta u = |H|^2 e^{2u}$  where $H \in A^2_1(\mathbb{D})$, introduced by Kraus \cite{kraus}.
At first glance, it seems that little is gained by working with $\GCE_H$ since it is equivalent to the usual Gauss curvature equation: $u$ is a solution of $\Delta u = |H|^2 e^{2u}$ if and only if
$v = u + \log |H|$ is a solution of 
\begin{equation}
\label{eq:GCE}
\Delta v = e^{2v} + 2\pi \sum_{c \in \mathcal Z(H)} \delta_c,
\end{equation}
where $\mathcal Z(H)$ denotes the zero set of $H$ counted with multiplicity. 
Nevertheless, it turns out that $\GCE_H$ is easier to work with than (\ref{eq:GCE}) since it features a holomorphic function instead of a measure.

We now give a brief summary of the argument in \cite{kraus} that the zero set of a function $H \in A^2_1(\mathbb{D})$ is a critical set of some Blaschke product (the other direction is trivial). Kraus first noticed that $\GCE_H$ has at least one solution, see Theorem \ref{kraus-PDE} below. In this case, $\GCE_H$ has a maximal solution $u_{H, \max}$ which dominates all other solutions pointwise.
Liouville's theorem implies that
$$
u_{H, \max}(z) = \log \frac{1}{|H(z)|} \cdot \frac{2|F'(z)|}{1-|F(z)|^2}
$$
for some holomorphic function $F: \mathbb{D} \to \mathbb{D}$, whose critical set contains the zero set of $H$. We can decompose $F = BSO$ into a product of a Blaschke factor, a singular inner factor and an outer factor. In \cite{kraus}, Kraus gave a very elegant argument that uses the maximality of $u_{H, \max}$ to rule out the existence of non-trivial singular and outer factors. Incidentally, $F$ is the maximal Blaschke product with critical set $C$ alluded to earlier.

In this paper, we will give another proof of Kraus' theorem based on the notion of the {\em canonical solution} of $\GCE_H$. The advantage of the canonical solution over the maximal solution is that it records more information about $H$ than its critical set. As the reader may guess, the canonical solution depends on the invariant subspace $[H] \subset A^2_1(\mathbb{D})$.

\section{Background in PDE}

In this section, we study the weighted Gaussian curvature equation ($\GCE_H$)
\begin{equation}
\label{eq:generalized-GCE2a}
   \left\{\begin{array}{lr}
        \Delta u  =  |H|^2 e^{2u} &  \text{in } \mathbb{D}, \\
       u = h, &  \text{on } \partial \mathbb{D},
        \end{array}\right.
\end{equation}
for $H \in A^2_1(\mathbb{D}) \setminus \{0\}$.
For a function $u$ to be considered a {\em solution}, we require:

\begin{enumerate}
\item For any $\phi \in C_c^\infty(\mathbb{D})$,
\begin{equation}
\label{eq:fbe-weak}
\int_{\mathbb{D}}  u \Delta \phi \, |dz|^2 =  \int_{\mathbb{D}}  e^{2u} \phi \, |dz|^2.
\end{equation}

\item As $r \to 1$, the measures
$u(re^{i\theta}) d\theta \to h \, d\theta$ converge in the weak-$*$ topology. (Unless otherwise specified, we interpret boundary values in this way.)
\end{enumerate}

By analogy with subharmonic functions, we say that $u$ is a {\em subsolution} if $\Delta u \ge |H|^2e^{2u}$ in the sense of distributions and a {\em supersolution}  if $\Delta u \le |H|^2e^{2u}$.

\begin{theorem}
\label{kraus-PDE}
For $H \in A^2_1(\mathbb{D})$, (\ref{eq:generalized-GCE2a}) 
admits a unique solution for any $h \in L^\infty(\partial \mathbb{D})$. 
The solution is increasing in $h$: if $h_1 \le h_2$, then $u_1 \le u_2$.
\end{theorem}

Uniqueness and monotonicity follow from Kato's inequality
  \cite[Proposition 6.9]{ponce-book} which states that {\em if $u \in L^1_{\loc}$ and $\Delta u \ge f$ in the sense of distributions with $f \in L^1_{\loc}$, then
$\Delta u^+ \ge f \cdot \chi_{u > 0}$}.
 As usual, $u^+ = \max(u,0)$ denotes the positive part of $u$.

\begin{proof}[Proof of Theorem \ref{kraus-PDE}: uniqueness and monotonicity]
By Kato's inequality,
$$
\Delta (u_1-u_2)^+ \ge  |H|^2 ( e^{2 u_1} - e^{2 u_2}) \cdot \chi_{\{u_1 > u_2\}} \ge 0
$$
is a subharmonic function. The inequality $h_1 \le h_2$ implies that $(u_1-u_2)^+$ has zero boundary values. The maximal principle shows that $(u_1-u_2)^+ \le 0$ or $u_1
 \le u_2$. The same argument also proves uniqueness.
\end{proof}

Existence is a standard application of Schauder's fixed point theorem, although for the convenience of the reader, we spell out the details.
As usual, $G(z, \zeta) =  \log  \bigl | \frac{1-z\overline{\zeta}}{z-\zeta} \bigr |$ denotes the Green's function of the unit disk. 
 We say that a measure $\mu$ on the unit disk is a {\em Blaschke measure}   if $(1-|z|^2) d\mu(z)$ is a finite measure. The following lemma is well-known:

\begin{lemma}
\label{compactness-lemma}
If $\mu$ is a Blaschke measure on the unit disk, then
$$
G_\mu(z) = \frac{1}{2\pi}  \int_{\mathbb{D}} G(z, \zeta) d\mu(\zeta)
$$
lies in the Sobolev space $W^{1,1}_0(\mathbb{D})$ and solves the linear Dirichlet problem
\begin{equation}
\label{eq:LDP}
   \left\{\begin{array}{lr}
        \Delta u = - \mu, &  \text{in } \mathbb{D}, \\
       u = 0, &  \text{on } \partial \mathbb{D},
        \end{array}\right.
\end{equation}
where the boundary condition is understood in terms of weak limits.
\end{lemma}

The $W^{1,1}$ regularity can be obtained in a number of ways. For a proof using Stampacchia's truncation method, see  \cite[Chapter 5]{ponce-book}. One can also
give a direct proof using the estimates
$$
\int_{\mathbb{D}} G(z, \zeta) |d\zeta|^2 \le C(1 - |z|)^2
$$
and
$$
\int_{\mathbb{D}} \bigl |\partial_z G(z, \zeta)\bigr | \, |d\zeta|^2 \le C(1 - |z|).
$$
The first estimate implies that $G_\mu \in L^{1}(\mathbb{D})$ while the second estimate implies that $\partial_z G_\mu \in L^1(\mathbb{D})$. Since $G$ is real-valued, we also have $\partial_{\overline{z}} G_\mu \in L^1(\mathbb{D})$.

\begin{proof}[Proof of Theorem \ref{kraus-PDE}:  existence]
Let $P_h$ denote the harmonic extension of $h$ to the unit disk. Since $h: \partial \mathbb{D} \to \mathbb{R}$ is bounded above by assumption, $P_h$ is bounded above on the unit disk.
Consider the closed convex set
$$\mathscr K_h \, = \, \Bigl \{v \in L^1(\mathbb{D}, |dz|^2), \  v \le P_h \Bigr \} \, \subset \,  L^1(\mathbb{D}, |dz|^2)$$
and the operator
\begin{equation}
\label{eq:schauder-operator}
(Tv)(z) = P_h(z) - \frac{1}{2\pi} \int_{\mathbb{D}} e^{2v(\zeta)} |H(\zeta)|^2 G(z, \zeta) |d\zeta|^2.
\end{equation}
The condition  $H \in A^2_1(\mathbb{D})$ tells us that  $d\mu_v(\zeta) = e^{2v(\zeta)} |H(\zeta)|^2 |d\zeta|^2$ is a Blaschke measure. Lemma \ref{compactness-lemma} tells us that integral in (\ref{eq:schauder-operator}) lies in $W^{1,1}_0(\mathbb{D})$. Since $W^{1,1}_0(\mathbb{D})$ sits compactly inside $L^1(\mathbb{D})$, the operator $T$ is compact and maps $\mathscr K_h$ into itself. By Schauder's fixed point theorem, $T$ has a fixed point. Applying Lemma \ref{compactness-lemma} again, we see that any fixed point of $T$ solves $\GCE_H$. The proof is complete.
\end{proof}

For future reference, we record:

\begin{lemma}[Comparison principle]
Suppose $u$ is a subsolution  and $v$ is a supersolution. If $u \le v$ on the boundary, then $u \le v$ in the interior. Here, one interprets ``\,$u \le v$ on $\partial \mathbb{D}$'' as 
$$
\limsup_{r \to 1} \int_0^{2\pi} \phi (e^{i\theta}) \bigl (u(re^{i\theta}) - v(re^{i\theta}) \bigr ) d\theta \le 0, \qquad \phi \in C(\partial \mathbb{D}), \quad \phi \ge 0.$$
\end{lemma}

\section{Canonical solutions}

Given a non-zero function $H \in A_1 ^2(\mathbb{D})$, let $u_{H,n}$ be the solution of the boundary value problem
\begin{equation}
\label{eq:H-GCE}
   \left\{\begin{array}{lr}
        \Delta u = |H|^2 e^{2u}, &  \text{in } \mathbb{D}, \\
       u = n, &  \text{on } \partial \mathbb{D}.
        \end{array}\right.
\end{equation}
By Liouville's theorem, $$u_{H, n} = \log \frac{1}{|H|} \frac{2|I_n'|}{1-|I_n|^2}$$ for some holomorphic self-map $I_n$ of the unit disk. From the monotonicity of solutions, we know that the functions $u_{H,n}$ are increasing in $n$.
Taking $n \to \infty$, we obtain 
 the {\em canonical solution} 
 \begin{equation}
 u_{H, \infty} := \lim_{n \to \infty} u_{H,n} = \log \frac{1}{|H|} \frac{2|I'|}{1-|I|^2}.
 \end{equation}
From the above construction, it is clear that the canonical solution is the minimal solution which dominates all solutions with finite boundary data.

Our main theorem states:

\begin{theorem}
\label{bounded-generator-thm}
The function $I$ is an inner function.
The invariant subspace $[H]$ is generated by $I'_n$ for any $n \in \mathbb{R}$ and contains $I'$.
\end{theorem}

\begin{remark}
It is well-known that $\GCE_H$ has a maximal solution, which dominates any other solution pointwise. It may come as a surprise to the reader that the canonical and maximal solutions may be different.
For instance, if $H$ has no zeros in the disk but generates a non-trivial invariant subspace in $A^2_1(\mathbb{D})$, then the maximal solution $u_{H, \max} = \log \frac{1}{|H|} \frac{2}{1-|z|^2}$ has the Liouville function $I_{H, \max}(z) = z$ but $I'_{H, \max} = 1 \notin [H]$. In Section \ref{sec:can-max}, we will show that the canonical and maximal solutions coincide if and only if $I$ is a maximal Blaschke product.
\end{remark}

For the convenience of the reader, we have split up the proof of Theorem \ref{bounded-generator-thm} into a long sequence of lemmas. In order to verify  Conjecture \ref{main-conjecture}, one would also need to show that $I'$ generates $[H]$. In Section \ref{sec:conj-proof}, we will reduce
 Conjecture \ref{main-conjecture} to a statement in PDE, which we hope is more tractable.

\subsection{\texorpdfstring{Why is $I' \in [H]$?}{Why is I' in H?}}

\begin{lemma}
\label{lem32}
 If a solution $u = \log \frac{1}{|H|} \frac{2|F'|}{1-|F|^2}$ is bounded above, then $F' \in [H]$.
\end{lemma}

\begin{proof}
Post-composing with a M\"obius tranformation in $\aut(\mathbb{D})$ if necessary, we may assume that $F(0) = 0$.
Since $u$ is subharmonic,
$$
u_{H,n} \le \sup u \text{ on } \mathbb{D} \ \implies \  |F'/H| \le e^{\sup u}/2 \ \implies \ F' \in [H],
$$
where we used that invariant subspaces of $A^2_1(\mathbb{D})$ are closed under multiplication by bounded holomorphic functions. 
\end{proof}

\begin{lemma}
\label{lem33}
Suppose $\{F_n\} \subset \hol(\mathbb{D}, \mathbb{D})$  converges uniformly on compact subsets to $F \in \hol(\mathbb{D}, \mathbb{D})$. If each $F'_n \in [H]$, then $F' \in [H]$.
\end{lemma}

\begin{proof}
While $F_n$ may not converge to $F$ in $A^2_1$, the Littlewood-Paley identity (\ref{eq:lp}) tells us that
$$\| F'_n \|_{A^2_1} \asymp \| F_n \|_{H^2} \le 1$$ are uniformly bounded. This allows us to pass to a subsequence that converges weakly. It remains to use the following fact from functional analysis: {\em in a Banach space, a subspace is closed if and only if it is weakly closed.}
\end{proof}

\begin{corollary}
The Liouville map $I_H$ associated to the canonical solution $u_{H, \infty}$ satisfies $I_H' \in [H]$.
\end{corollary}

\subsection{Why is $I$ an inner function?}

\begin{lemma}
When $H = I'$ is the derivative of an inner function, we can write down the canonical solution explicitly: 
$$
u_{I', \infty} = \log \frac{1}{|I'|} \frac{2|I'|}{1-|I|^2} = \log \frac{2}{1-|I|^2}.
$$
\end{lemma}

\begin{proof}
($\ge$) To show that $\log \frac{2}{1-|I|^2} \ge u_{I', \infty}$, we check that $\log \frac{2}{1-|I|^2} \ge u_{I', n}$ for any $n \in \mathbb{R}$. To that end, we verify that $v_n = \bigl [u_{I',n} - \log \frac{2}{1-|I|^2} \bigr ]^+ = 0$. 
By Kato's inequality, $\Delta v_n \ge 0$ is subharmonic. Since $I$ is inner, $v_n$ has zero boundary values (in terms of weak-$*$ limits of measures). Since $v_n$ is non-negative, it is identically 0.

($\le$) Conversely, one can approximate $ \log \frac{2}{1-|I|^2}$ by the solutions
$
 \log \frac{1}{|I'|} \frac{2|rI'|}{1-|rI|^2} = \log \frac{2r}{1-|rI|^2}
$
that are bounded above. Since $\log \frac{2r}{1-|rI|^2} \le u_{I', \infty}$ lies below the canonical solution for any $0 < r < 1$, so must $\log \frac{2}{1-|I|^2} \le u_{I', \infty}$.
\end{proof}

\begin{lemma}
\label{lem36}
The Liouville map $I_H$ associated to a canonical solution $u_{H, \infty}$ is an inner function.
\end{lemma}

\begin{proof}
Since $I' \in [H]$, there exists a sequence of polynomials $p_k$ such that $Hp_k \to I'$ in $A^2_1(\mathbb{D})$. Replacing $p_k(z)$ by $p_k(rz)$ if necessary, we may assume that the $p_k$ have no zeros on the unit circle.
By the comparison principle,
\begin{equation}
\label{eq:comp}
u_{H p_k, n}
\le  \log \frac{1}{|Hp_k|} \cdot  \frac{2|I'|}{1-|I|^2},
\end{equation}
 for any $n \in \mathbb{R}$. Note that if $p_k$ has zeros inside the disk, then the RHS will be a supersolution of $\GCE_{Hp_k}$ rather than a solution, which makes it easier for (\ref{eq:comp}) to hold.
Taking $k \to \infty$ in (\ref{eq:comp}), we get 
$$
u_{I', n} \le  \log  \frac{2}{1-|I|^2}.
$$ 
Since this is true for any $n \in \mathbb{R}$, 
$
 u_{I', \infty} \le \log \frac{2}{1-|I|^2},
$
which forces $I$ to be inner.
\end{proof}

\begin{lemma}
Inner functions embed into invariant subspaces: $$\inn / \, \aut(\mathbb{D}) \subseteq\bigl \{ \text{1-generated invariant subspaces of }A^2_1 \bigr \}.$$
\end{lemma}

\begin{proof}
 The proof of Lemma \ref{lem36} shows that if $[H_1] \subset [H_2]$, then
$$
\frac{|I'_{H_1}|}{1-|I_{H_1}|^2} \le \frac{|I'_{H_2}|}{1-|I_{H_2}|^2}.
$$
In particular, if $H_1$ and $H_2$ generate the same invariant subspace, then
$$
\frac{|I'_{H_1}|}{1-|I_{H_1}|^2} = \frac{|I'_{H_2}|}{1-|I_{H_2}|^2}.
$$
By Liouville's theorem, $I_{H_1} = I_{H_2}$ up to post-composition with an element of $\aut(\mathbb{D})$. The proof is complete.
\end{proof}
 
\subsection{\texorpdfstring{Why does $I'_n$ generate $[H]$?}{Why does I'\_n generate [H]?}}

\begin{proof}[Proof of Theorem \ref{bounded-generator-thm}]
For concreteness, we will show that $I'_0$ generates $[H]$ as the general case is similar.
Since the subharmonic function
$
u = \log \frac{1}{|H|} \frac{2 |I_0'|}{1-|I_0|^2}
$
has zero boundary data, its  minimal harmonic majorant is 0.
For $0 < r < 1$, define
$$
u_r = \log \frac{1}{|H_r|} \frac{2 |r I_0'|}{1-|r I_0|^2},
$$
where
 $$H_r := \frac{2r}{\phi_r}  \cdot I'_0\in [I'_0], \qquad  \phi_r(z) := \out_{1-|r I_0|^2}(z).$$
Here, $\phi_r$ is the outer function with absolute value $1-|r I_0|^2$ on the unit circle.
 By construction, the  minimal harmonic majorant of $u_r$ is also 0.
 
 Since the $|H_r|$ increase to $|H|$, after passing to a subsequence, the functions $H_r \to e^{i\theta} H$ converge in $A^2_1(\mathbb{D})$ for some $\theta \in [0, 2\pi)$. It follows that $H \in [I'_0]$ and therefore, $[H] = [I'_0]$.
\end{proof}

\subsection{Does $I'$ generate $[H]$?}
\label{sec:conj-proof}

We now deduce Conjecture \ref{main-conjecture} from a statement in PDE that looks intuitively plausible:

\begin{conjecture}
\label{conjecture2}
Any solution of $\Delta u =|I'|^2 e^{2u}$ that is $\le u_{I', \infty}$ can be approximated uniformly on compact subsets by solutions $u_k$ that are bounded above.
\end{conjecture}

\begin{proof}[``Proof'' of Conjecture \ref{main-conjecture}]
Let $H \in A^2_1(\mathbb{D})$ and $I_0$ and $I$ be the Liouville functions of $u_{H,0}$ and $u_{H, \infty}$ respectively. By the above ``fact'' we can approximate
\begin{equation*}
u_k \, = \, \log \frac{1}{|I'|} \frac{|F_k'|}{1-|F_k|^2} \, \to \, \log \frac{1}{|I'|} \frac{|I_0'|}{1-|I_0|^2}.
\end{equation*}
Since the solutions $u_k$ are bounded above, each $F'_k \in [I']$. By Lemma \ref{lem33}, $I'_0 \in [I']$. Therefore, $[I'] = [H]$ as desired.
\end{proof}

\begin{remark}
Conjecture \ref{conjecture2} is easy to believe since one can approximate a harmonic function by bounded harmonic functions. However, it turns out to be quite difficult and we can only verify it in special cases, for instance, when $I'$ lies in Nevanlinna class.
\end{remark}

\subsection{When are canonical solutions maximal?}
\label{sec:can-max}

Since solutions of $\GCE_H$ are in bijection with solutions of (\ref{eq:GCE}), the maximal solution of $\GCE_H$ depends only on the zero set of $H$. A little thought shows that
$$
u_{H, \max} =  \log \frac{1}{|H|} \frac{|F_{\mathcal Z(H)}|}{1-|F_{\mathcal Z(H)}|^2},
$$
where $F_{\mathcal Z(H)}$ denotes the maximal Blaschke product whose critical set is the zero set of $H$.

\begin{lemma}
For $H \in A^2_1(\mathbb{D})$, the canonical and maximal solution coincide if and only if $I_H \in \MBP$ where $I_H$ is the Liouville map of the canonical solution associated to $H$.
\end{lemma} 

\begin{proof}The canonical solution
$
u_{H, \infty} =  \log \frac{1}{|H|} \frac{|I'_H|}{1-|I_H|^2}
$
is equal to the maximal solution if and only if
$$
\frac{|I'_H|}{1-|I_H|^2} = \frac{|F'_{\mathcal Z(H)}|}{1-|F_{\mathcal Z(H)}|^2}.
$$
By Liovuille's theorem, this happens precisely when $I_H$ and $F_{\mathcal Z(H)}$ are related by an element of $\aut(\mathbb{D})$.
\end{proof}

\subsection{Boundary behaviour of canonical solutions}

In \cite{maximal-blaschke}, Kraus and Roth showed that a maximal Blaschke product $F_C$ extends analytically past any arc on the unit circle which does not meet the closure of $C$. We generalize their result to canonical solutions:

\begin{theorem}
If $H \in A^2_1(\mathbb{D})$ extends analytically past an open arc $J \subset \partial \mathbb{D}$, then the canonical solution $u_{H,\infty}(z) \to +\infty$ as $z$ approaches $J$. In particular, $I_H$ extends analytically though $J$.
\end{theorem}

\begin{proof}
In the proof of Theorem \ref{kraus-PDE}, we saw that
$$ u_n(z) = n - \frac{e^{2n}}{2\pi} \int_{\mathbb{D}}  |H(\zeta)|^2 G(z, \zeta) |d\zeta|^2.$$
From this representation, it follows that for any $\zeta \in J$,
$$
\liminf_{z \to \zeta} u_{H, n}(z) \ge n.
$$
Since the $u_{H,n}$ are increasing in $n$,
$$
\lim_{z \to \zeta} u_{H, \infty}(z) = \infty.
$$
Since $H(z)$ is bounded near $\zeta$, 
$$
\log \frac{|I'(z)|}{1-|I(z)|^2}
$$
tends to $+\infty$ as $z \to \zeta$. The theorem now follows from \cite[Theorem 1.1]{KRR}.
\end{proof}

\section{Further remarks and open problems}

We conclude with some remarks and open problems:

\begin{enumerate}
\item In the theory of Bergman spaces, one learns that any 1-generated invariant subspace $X \subset A^2_1$ can be generated by a Bergman inner function $\varphi = \varphi_X$, which solves a certain extremal problem. According to  \cite[Theorem 3.3]{HKZ},  $\varphi$ is the derivative of a BMO function. Since one expects $\varphi$ to be the smoothest function in $X$, it is natural to wonder if $\varphi$ is actually the derivative of a bounded function. 

\item Can one give an alternative proof of Shimorin's result on asymptotic spectral synthesis \cite{shimorin} in $A^2_1$ using the methods of this paper? That is, to show that one can approximate any 1-generated invariant subspace $X \subset A^2_1$ by zero-based ones.

\item For $\alpha > -1$, the weighted Bergman space $
A_\alpha^p$ consists of functions for which
 \begin{equation*}
 \|f\|_{A^p_\alpha} = \biggl (  \int_{\mathbb{D}} |f(z)|^p  (1-|z|)^\alpha |dz|^2 \biggr)^{1/p} < \infty.
\end{equation*}
When one does not write the subscript, one takes $\alpha = 0$. It is well known that the zero sets of all Bergman spaces $A^p$ are different.  Density conditions are known to separate zero sets of $A^p_\alpha$ and $A^q_\beta$ if $\frac{\alpha+1}{p} \ne \frac{\beta+1}{q}$, see \cite{HKZ}.
Do $A^p_{\alpha}$ and $A^q_{\beta}$ have the same zero sets if $\frac{\alpha+1}{p} = \frac{\beta+1}{q}$?

\end{enumerate}

%(ii)  To see that $I_H$ extends analytically though $J$, it is enough to show that $|I_H(z)| \to 1$ as $z \to \zeta$ with $\zeta \in J$. The proof involves hyperbolic geometry.

\bibliographystyle{amsplain}

\end{document}